\documentclass{amsart}

\usepackage{mathtools}
\usepackage{rotating}
\usepackage{amssymb}
\usepackage{tikz}
\usetikzlibrary{calc}
\usetikzlibrary{decorations}
\usepackage{todonotes}
\usepackage{hyperref}
\usepackage{stackengine}
\usepackage{mathabx}

\hypersetup {
    colorlinks = true,
    linktoc = all,
    linkcolor = blue,
    citecolor = red
}

\newtheorem{theorem}{Theorem}[section]
\newtheorem{lemma}[theorem]{Lemma}

\newtheorem{question}[theorem]{Question}

\theoremstyle{definition}

\newtheorem{example}[theorem]{Example}

\theoremstyle{remark}

\numberwithin{equation}{section}

\begin{document}

\title[Reddening Sequences and Mutation of infinite quivers]{Reddening Sequences and Mutation of infinite quivers}


\author{Eric Bucher}
\address{Department of Mathematics, Xavier University, Cincinnati, OH 45207 }
\email{buchere1@xavier.edu}

\author{Elizabeth Howard}
\address{Department of Mathematics, Xavier University, Cincinnati, OH 45207 }
\email{howarde6@xavier.edu}





\begin{abstract}
Cluster algebras, introduced by Fomin and Zelevinsky \cite{cluster1} through the process of quiver mutation, have become central objects in modern algebra and geometry, linking combinatorial constructions with diverse mathematical domains such as Teichmüller theory, total positivity, and even theoretical physics. Building on foundational work by Fomin, Shapiro, and Thurston \cite{F-S-T} connecting cluster algebras to triangulated surfaces, recent research has extended mutation theory to infinite settings, including the infinity-gon and more general marked surfaces. In this paper, we develop a purely combinatorial framework for mutation of infinite quivers, independent of but compatible with these topological constructions. By formalizing infinite quivers as limits of embedded finite quivers, we establish a consistent definition of mutation that generalizes prior surface-based results. We then apply this framework to extend the notion of reddening sequences—special mutation sequences with significant algebraic consequences—from the finite to the infinite setting. Our approach not only unifies previous topological and combinatorial perspectives but also provides a technical foundation for further generalizations of cluster algebra theory in the infinite case.
\end{abstract}

\maketitle

\section{Introduction}

Quiver mutation was introduced by Fomin and Zelevinsky \cite{cluster1} as a way to create a combinatorial object called a cluster algebra. The cluster algebra is given by a quiver along with a set of variables attached to each vertex called a seed, $(Q,\mathbf{x})$. Mutation gives a process of creating a new quiver and seed pair from the previous. Applying mutation in all possible directions and taking the union of all of these seeds provides a generating set for what is called the cluster algebra associated to the pair $(Q,\mathbf{x})$. In the last twenty years, cluster algebras have proven to be extremely useful and varied when considering different fields of mathematics. They have applications to totally non-negative matrices, Teichmüller theory, scattering diagrams, and connections to predicting results of high-energy particle collisions, to name a few.

 Fomin, Shapiro and Thurston \cite{F-S-T} considered cluster algebras that arise from triangulations of marked surfaces in the finite setting. They classified these cluster algebras early on, giving a framework for a very important family. The algebras in question have many desirable properties and applications. In this topological landscape, the researchers began their exploration as to what mutation would look like for a quiver that had infinite vertices and infinite edges. Ndouné \cite{ndoune} was the first person to do this, for what they called the infinity-gon. Ndouné considered a surface with one boundary component with an infinite number of marked points. This work was extended in the works of Baur and Gratz \cite{Baur-Gratz} and further more by Çanaçki and Felikson \cite{C-F} to more general surfaces. In the case of Baur and Gratz, they considered additional boundary components; in the Çanaçki and Felikson work, they moved to the fully generalized surfaces setting. The most recent work is by Çanaçki, Kalck, and Pressland \cite{C-K-P} which considers the categorifications of the infinity-gon. All of this work paints a unified picture of what mutation should look like when considering this from the topological perspective. 

In this paper, we seek to use this topological framework to motivate a purely combinatorial approach to mutation of infinite quivers. That is to say, we want a general definition of an infinite quiver and mutation of such an object so that in the case where that quiver does arise from a topological setting, it is consistent with the work mentioned above. This would allow a generalization of the wealth of work to be extended from the generalized cluster algebras from surfaces setting to the arbitrary infinite quiver setting. We see the technical aspects of this paper to be a useful tool for others to generalize their results.

In section 3, we will present some examples of infinite quivers and what we would expect mutation to look like in those scenarios. Then we will formally define infinite quivers using a sequence of embedded finite quivers and subsequently mutation on an infinite quiver. This will provide the technical results necessary to show that this definition is consistent. 

In section 4, we will consider an application of these ideas. Reddening sequences were first brought to us in the context of cluster algebras by Keller \cite{keller} and then investigated by Brüstle, Dupont and Pérotin \cite{B-D-P} in their paper \emph{On Maximal Green Sequences}. These are a special sequence of mutations that, in the finite setting, guarantee your cluster algebra has certain desirable characteristics. In general, the exploration of reddening sequences has been well studied by many groups of researchers: 

\begin{itemize}
    \item Muller \cite{Muller} explored the connection between reddening sequences and mutation equivalent quivers;
    \item Garver and Musiker \cite{G-M} considered type $A$ cluster algebras;
    \item Gonchorov and Shen \cite{shen} found reddening sequences for cluster algebras from surfaces;
    \item B., Mills \cite{bucher}, \cite{bucher-mills} looked at cluster algebras from surfaces;
    \item B., M. \cite{banff}, \cite{cp} explored connections to reddening sequences and Banff quivers.
\end{itemize}

We will extend this notion of reddening sequences to infinite quivers, providing a definition consistent with the finite setting. This definition is very functional, as any infinite quiver with a reddening sequences results in a family of reddening sequences for each finite subquiver. We will then prove some results about families of infinite quivers, which do in fact have infinite reddening sequences. 

But before we do any of that work, section 2 will now provide the background on quiver mutation and reddening sequences we will need for the later sections.

\section{Background: Quiver mutation and reddening sequences}\label{sec:mu}
A \textbf{\emph{quiver}} is a pair $Q = (V,E)$ where $V$ is a set of \emph{vertices} and $E$ is mulitset of \textbf{\emph{arrows}} between two distinct vertices.
That is, if $i \to j$ for $i,j \in V$ is an arrow then we require that $i \neq j$ (i.e. no loops).
Furthermore, we do not allow both $i \to j \in E$ and $j \to i \in E$ (i.e. no $2$-cycles).
Let us point out that $E$ is a multiset, meaning multiple arrows between two particular vertices is permitted. 

The \textbf{\emph{framed quiver}} associated to $Q$, denoted $\widehat{Q}$, is the quiver where the vertex set and edge set are the following: $V(\widehat{Q}):=V(Q)\sqcup\{i'\space|\space i\in V(Q)\}$ and $E(\widehat{Q}):=E(Q)\sqcup\{i\rightarrow i'\space |\space i\in V(Q)\}$. Similarly, the \textbf{\emph{coframed quiver}} associated to $Q$, denoted $\widecheck{Q}$, is the quiver where the vertex set and edge set are the following: $V(\widecheck{Q}):=V(Q)\sqcup\{i'\space|\space i\in V(Q)\}$ and $E(\widecheck{Q}):=E(Q)\sqcup\{i'\rightarrow i\space |\space i\in V(Q)\}$. We say that the vertices $i\in V(Q)$ are the \textbf{\emph{mutable vertices}}, while the vertices $i'$ are the frozen vertices. Note that mutation is not allowed at a frozen vertex. 

We are now ready to define the process of quiver mutation which is a combinatorial ingredient in Fomin and Zelevinsky’s cluster algebras~\cite{cluster1}. 
The \textbf{\emph{mutation}} of a quiver $Q$ at a mutable vertex $k$ is denoted $\mu_k(Q)$ and is obtained by
\begin{enumerate}
    \item adding an arrow $i \to j$ for any $2$-path of vertices $i \to k \to j$,
    \item reverse all arrows incident to $k$,
    \item delete a maximal collection of $2$-cycles as well as any arrow between frozen vertices.
\end{enumerate}

A mutable vertex is \textbf{\emph{green}} if there are no incident incoming arrows from frozen vertices. Likewise, a mutable vertex is \textbf{\emph{red}} if there are no incident outgoing arrows to frozen vertices. Notice that all mutable vertices in framed quiver $\widehat{Q}$ are green. A \textbf{\emph{reddening sequence}} is a sequence of mutations that takes all mutable vertices from green to red. 

\section{Infinite Quivers and Mutation}

In this section, we start by using an example from the topological framework provided by Çanaçki and Felikson \cite{C-F} to motivate the resulting formal definition of infinite quivers and mutation. Figure \ref{fig:surface} shows a triangulation of a surface resulting in an infinite fan taken from Çanaçki and Felikson \cite{C-F}. This is one of two ways a limit arc can arise, and we will use this example to illustrate how we want infinite quivers and mutation to behave. The natural and obvious quiver that arises from the triangulated surface on the left of Figure \ref{fig:surface} is pictured in the top of Figure \ref{fig:surfacequiver}, denoted $Q_\infty$. This quiver is an infinite path in one direction. When an arc is flipped in the triangulation, as seen on the right of Figure \ref{fig:surface}, the quiver that arises from this resulting surface is shown on the bottom of Figure \ref{fig:surfacequiver}. It seems that mutation of $Q_\infty$ at vertex $3$ results in the quiver at the bottom of Figure \ref{fig:surfacequiver}. The goal of this section is to remove the need to consider the topological object at all, and just perform mutation on the infinite quiver itself. This can cause some technical issues, thus we need a careful and formal definition for both an infinite quiver and mutation on this mathematical object.

\begin{figure}[h!]
\centering

\begin{minipage}{0.45\textwidth}
\centering
\begin{tikzpicture}[scale=1]

    \coordinate (A) at (0,4);
    \fill (A) circle (3pt);

    \coordinate (B1) at (-2.8,0);
    \coordinate (B2) at (-2.1,0);
    \coordinate (B3) at (-1.05,0);
    \coordinate (B4) at (0,0);
    \coordinate (B5) at (0.84,0);

    \coordinate (F1) at (1.3,0.4);
    \coordinate (F2) at (1.82,0.4);

    \coordinate (Copen) at (2.52,0);

    \fill[cyan!20] (B1) -- (A) -- (Copen) -- cycle;

    \draw[blue,thick] (A) -- (B1);
    \draw[blue,thick] (A) -- (B2);
    \draw[blue,thick] (A) -- (B3);
    \draw[blue,thick] (A) -- (B4);
    \draw[blue,thick] (A) -- (B5);

    \draw[blue,thick] (A) -- (F1);
    \draw[blue,thick] (A) -- (F2);

    \draw[blue,thick,dashed] (A) -- (Copen);

    \node at (1.6,0.3) {\textbf{\dots}};

    \foreach \P in {B1,B2,B3,B4,B5}
        \fill (\P) circle (2pt);

    \draw (Copen) circle (3pt);

\end{tikzpicture}
\end{minipage}
\hfill
\begin{minipage}{0.45\textwidth}
\centering
\begin{tikzpicture}[scale=1]

    \coordinate (A) at (0,4);
    \fill (A) circle (3pt);
    \coordinate (B1) at (-2.8,0);
    \coordinate (B2) at (-2.1,0);
    \coordinate (B3) at (-1.05,0);
    \coordinate (B4) at (0,0);
    \coordinate (B5) at (0.84,0);
    \coordinate (F1) at (1.3,0.4);
    \coordinate (F2) at (1.82,0.4);
    \coordinate (Copen) at (2.52,0);

    \fill[cyan!20] (B1) -- (A) -- (Copen) -- cycle;
    \draw[blue,thick] (A) -- (B1);
    \draw[blue,thick] (A) -- (B2);
    \draw[blue,thick] (A) -- (B4);
    \draw[blue,thick] (A) -- (B5);
    \draw[blue,thick] (A) -- (F1);
    \draw[blue,thick] (A) -- (F2);
    \draw[blue,thick] (B2) to[out=60, in=120] (B4);
    \draw[blue,thick,dashed] (A) -- (Copen);
    \node at (1.6,0.3) {\textbf{\dots}};
    \foreach \P in {B1,B2,B3,B4,B5}
        \fill (\P) circle (2pt);
    \draw (Copen) circle (3pt);

\end{tikzpicture}
\end{minipage}

\caption{The left shows a triangulated surface where the limit arc arises in an infinite fan, while the right shows this triangulation with a singular flipped arc.}
\label{fig:surface}
\end{figure}

\begin{figure}
    \centering
    \footnotesize
    \begin{tikzpicture}[scale=0.3]
        \node (1) at (0,0) {$1$};
        \node (2) at (3,0) {$2$};
        \node (3) at (6,0) {$3$};
        \node (4) at (9,0) {$4$};
        \node (5) at (12,0) {$5$};  
        \node (6) at (15,0) {$\cdots$};
        \node (7) at (-4,0) {$Q_\infty$:};
        
        \draw[-{latex}] (1) to (2);
        \draw[-{latex}] (2) to (3);
        \draw[-{latex}] (3) to (4);
        \draw[-{latex}] (4) to (5);
        \draw[-{latex}] (5) to (6);

        \node (11) at (0,-2) {$1$};
        \node (22) at (3,-2) {$2$};
        \node (33) at (6,-2) {$3$};
        \node (44) at (9,-2) {$4$};
        \node (55) at (12,-2) {$5$};  
        \node (66) at (15,-2) {$\cdots$};
        \node (77) at (-4,-2) {$Q_\infty'$:};

        \draw[-{latex}] (11) to (22);
        \draw[-{latex}] (33) to (22);
        \draw[-{latex}] (44) to (33);
        \draw[-{latex}, bend left = 25] (22) to (44);
        \draw[-{latex}] (44) to (55);
        \draw[-{latex}] (55) to (66);
    \end{tikzpicture}
    \caption{The infinite quiver $Q_\infty$ arises from the triangulation of the surface on the left of Figure \ref{fig:surface}, and $Q_\infty '$ arises from the triangulation of the surface on the right of the same figure.}
    \label{fig:surfacequiver}
\end{figure}

\subsection{Infinite quivers.}

We define an \textbf{\emph{infinite quiver}} as a sequence of quivers $\{Q_i\}$ with a sequence of embeddings $\{\tau_i\}$ such that $Q_i$ is an induced subquiver of $Q_{i+1}$ for all $i$. We denote this $Q_\infty$ and we say $Q_\infty$ is the limit of $\{Q_i, \tau_i\}$. A vertex $v$ is in $Q_\infty$ when $v$ is in $Q_i$ for some $i$. Similarly, we say an arrow $x\rightarrow y $ is in $Q_\infty$ when $x\rightarrow y $ is in $Q_i$ for some $i$. It is often easier to visualize this in the following way:

$$Q_{\infty} := Q_1\xhookrightarrow[\tau_1]{}  Q_2 \xhookrightarrow[\tau_2]{} Q_3\xhookrightarrow[\tau_3]{} \cdots \xhookrightarrow[\tau_{i-1}]{} Q_i \xhookrightarrow[\tau_i]{} \cdots$$

We consider the following two examples in Figures \ref{fig:inf1} and \ref{fig:inf2} of infinite quivers. These examples illustrate how it is important to not only consider the sequence of quivers, $\{Q_i\}$, but also the embeddings $\{\tau_i\}$ as these infinite quivers are distinct yet have isomorphic quiver sequences.

\begin{figure}
    \centering
        \footnotesize
        \begin{tikzpicture}[scale=0.3]
        \node (1) at (0,8) {$1$};
        \node (2) at (0,6) {$1$};
        \node (3) at (3,6) {$2$};
        \node (4) at (0,4) {$1$};
        \node (5) at (3,4) {$2$};
        \node (6) at (6,4) {$3$};
        \node (7) at (0,2) {$1$};
        \node (8) at (3,2) {$2$};
        \node (9) at (6,2) {$3$};
        \node (10) at (9,2) {$4$};
        \node (11) at (0,0) {$1$};
        \node (12) at (3,0) {$2$};
        \node (13) at (6,0) {$3$};
        \node (14) at (9,0) {$4$};
        \node (15) at (12,0) {$5$};  
        \node (16) at (7.5,-2) {$\vdots$};
        \node (17) at (0,-4) {$1$};
        \node (18) at (3,-4) {$2$};
        \node (19) at (6,-4) {$3$};
        \node (20) at (9,-4) {$4$};
        \node (21) at (12,-4) {$5$};
        \node (22) at (15,-4) {$\cdots$};
        \node (23) at (18,-4) {$i$};
        \node (24) at (7.5,-6) {$\vdots$};
        \node (25) at (-3,8) {$Q_1$:};
        \node (26) at (-3,6) {$Q_2$:};
        \node (27) at (-3,4) {$Q_3$:};
        \node (28) at (-3,2) {$Q_4$:};
        \node (29) at (-3,0) {$Q_5$:};
        \node (29) at (-3,-4) {$Q_i$:};
        \node (30) at (-3,-8) {$Q_\infty$:};
        \node (31) at (0,-8) {$1$};
        \node (32) at (3,-8) {$2$};
        \node (33) at (6,-8) {$3$};
        \node (34) at (9,-8) {$4$};
        \node (35) at (12,-8) {$5$};
        \node (36) at (15,-8) {$\cdots$};
    
        \draw[-{latex}] (2) to (3);
    
        \draw[-{latex}] (4) to (5);
        \draw[-{latex}] (5) to (6);
    
        \draw[-{latex}] (7) to (8);
        \draw[-{latex}] (8) to (9);
        \draw[-{latex}] (9) to (10);

        \draw[-{latex}] (11) to (12);
        \draw[-{latex}] (12) to (13);
        \draw[-{latex}] (13) to (14);
        \draw[-{latex}] (14) to (15);
        
        \draw[-{latex}] (17) to (18);
        \draw[-{latex}] (18) to (19);
        \draw[-{latex}] (19) to (20);
        \draw[-{latex}] (20) to (21);
        \draw[-{latex}] (21) to (22);
        \draw[-{latex}] (22) to (23);

        \draw[-{latex}] (31) to (32);
        \draw[-{latex}] (32) to (33);
        \draw[-{latex}] (33) to (34);
        \draw[-{latex}] (34) to (35);
        \draw[-{latex}] (35) to (36);
        \end{tikzpicture}
    \caption{The sequence of quivers $Q_\infty=\{Q_i\}$ with $\tau_i(k)=k$.}
    \label{fig:inf1}
\end{figure}

\begin{figure}
    \centering
    \footnotesize
        \begin{tikzpicture}[scale=0.3]
        \node (1) at (0,8) {$0$};
        \node (2) at (0,6) {$0$};
        \node (3) at (3,6) {$1$};
        \node (4) at (0,4) {$-1$};
        \node (5) at (3,4) {$0$};
        \node (6) at (6,4) {$1$};
        \node (7) at (0,2) {$-1$};
        \node (8) at (3,2) {$0$};
        \node (9) at (6,2) {$1$};
        \node (10) at (9,2) {$2$};
        \node (11) at (0,0) {$-2$};
        \node (12) at (3,0) {$-1$};
        \node (13) at (6,0) {$0$};
        \node (14) at (9,0) {$1$};
        \node (15) at (12,0) {$2$};  
        \node (16) at (6,-2) {$\vdots$};
        \node (25) at (-3,8) {$Q'_1$:};
        \node (26) at (-3,6) {$Q'_2$:};
        \node (27) at (-3,4) {$Q'_3$:};
        \node (28) at (-3,2) {$Q'_4$:};
        \node (29) at (-3,0) {$Q'_5$:};
        \node (30) at (-3,-4) {$Q_\infty'$:};
        \node (31) at (0,-4) {$\cdots$};
        \node (32) at (3,-4) {$-1$};
        \node (33) at (6,-4) {$0$};
        \node (34) at (9,-4) {$1$};
        \node (35) at (12,-4) {$2$};
        \node (36) at (15,-4) {$\cdots$};
    
        \draw[-{latex}] (2) to (3);
    
        \draw[-{latex}] (4) to (5);
        \draw[-{latex}] (5) to (6);
    
        \draw[-{latex}] (7) to (8);
        \draw[-{latex}] (8) to (9);
        \draw[-{latex}] (9) to (10);

        \draw[-{latex}] (11) to (12);
        \draw[-{latex}] (12) to (13);
        \draw[-{latex}] (13) to (14);
        \draw[-{latex}] (14) to (15);

        \draw[-{latex}] (31) to (32);
        \draw[-{latex}] (32) to (33);
        \draw[-{latex}] (33) to (34);
        \draw[-{latex}] (34) to (35);
        \draw[-{latex}] (35) to (36);
        \end{tikzpicture}
    \caption{The infinite quiver $Q'_\infty$ is not isomorphic to $Q_\infty$ even though each individual $Q_i$ is isomorphic to $Q_i'$.}
    \label{fig:inf2}
\end{figure}

\subsection{Mutation of infinite quivers.}
We now want to define an analogue of mutation for infinite quivers. Let $k$ be a vertex in $Q_\infty$. Let $j$ be the minimal value such that $k\in Q_j$. Then, we define \textbf{\emph{mutation of $Q_{\infty}$ at vertex $k$}} as: $$\mu_k(Q_\infty): = \mu_k(Q_j)|_{V(Q_1)}, \mu_k(Q_j)|_{V(Q_2)}, \dots, \mu_k(Q_j)|_{V(Q_{j-1})}, \mu_k(Q_j), \mu_k(Q_{j+1}), \dots$$

Note that the resulting quiver uses the same sequence of embeddings $\{\tau_i\}$, i.e. the embedding map of the resulting mutated quiver is of the form $$\mu_k(Q_j)|_{V(Q_1)} \xhookrightarrow[\tau_1]{} \cdots \xhookrightarrow[\tau_{j-1}]{} \mu_k(Q_j)|_{V(Q_{i+1})} \xhookrightarrow[\tau_j]{} \mu_k(Q_j) \xhookrightarrow[\tau_{j}]{} \mu_k(Q_{j+1}) \xhookrightarrow[\tau_{j+1}]{} \cdots$$

We now demonstrate that this is properly defined and show that the result of mutation is an infinite quiver.

\begin{theorem}
   For an infinite quiver $Q_{\infty}$, we have that $\mu_k(Q_\infty)$ is also an infinite quiver. 
\end{theorem}

\begin{proof}
    Let $Q_\infty$ be an infinite quiver and let $k$ be a vertex in $Q_\infty$. Then define $\mu_k(Q_\infty)$ as above, with $j$ being the minimal value such that $k$ is in $Q_j$. For notation's sake, let $\mu_k(Q_\infty)_i:=\tilde{Q}_i$. We will show $\mu_k(Q_\infty)$ is an infinite quiver by showing $\tilde{Q}_i$ is an induced subquiver of $\tilde{Q}_{i+1}$ for all $i\in \mathbb{N}$. By definition, for all $i<j$, we have that $\tilde{Q}_i$ is an induced subquiver of $\tilde{Q}_{i+1}$. Now consider $i\geq j$. Note that by definition, $Q_t=Q_{t+1}|_{V(Q_t)}$ for all $t\in\mathbb{N}$. Hence, $\tilde{Q}_i=\mu_k(Q_i)=\mu_k(Q_{i+1}|_{V(Q_i)})$. Since mutation respects induced subquivers, we have that $\mu_k(Q_{i+1}|_{V(Q_i)}) = \mu_k(Q_{i+1})|_{V(Q_i)}$. Or in other words, $\tilde{Q}_i=\tilde{Q}_{i+1}|_{V(Q_i)}$. 
\end{proof}

\begin{figure}
    \centering
    \footnotesize
    \begin{tikzpicture}[scale=0.3]
        \centering
        \node (1) at (0,8) {$1$};
        \node (2) at (0,6) {$1$};
        \node (3) at (3,6) {$2$};
        \node (4) at (0,4) {$1$};
        \node (5) at (3,4) {$2$};
        \node (6) at (6,4) {$3$};
        \node (7) at (0,2) {$1$};
        \node (8) at (3,2) {$2$};
        \node (9) at (6,2) {$3$};
        \node (10) at (9,2) {$4$};
        \node (11) at (0,0) {$1$};
        \node (12) at (3,0) {$2$};
        \node (13) at (6,0) {$3$};
        \node (14) at (9,0) {$4$};
        \node (15) at (12,0) {$5$};  
        \node (16) at (7.5,-1.5) {$\vdots$};
        \node (17) at (0,-4) {$1$};
        \node (18) at (3,-4) {$2$};
        \node (19) at (6,-4) {$3$};
        \node (20) at (9,-4) {$4$};
        \node (21) at (12,-4) {$5$};
        \node (22) at (15,-4) {$\cdots$};
        \node (23) at (18,-4) {$i$};
        \node (24) at (7.5,-6) {$\vdots$};
        \node (25) at (-6,8) {$\mu_3(Q_3)|_{V(Q_1)}$:};
        \node (26) at (-6,6) {$\mu_3(Q_3)|_{V(Q_2)}$:};
        \node (27) at (-6,4) {$\mu_3(Q_3)$:};
        \node (28) at (-6,2) {$\mu_3(Q_4)$:};
        \node (29) at (-6,0) {$\mu_3(Q_5)$:};
        \node (29) at (-6,-4) {$\mu_3(Q_i)$:};

        \draw[-{latex}] (2) to (3);

        \draw[-{latex}] (4) to (5);
        \draw[-{latex}] (6) to (5);

        \draw[-{latex}] (7) to (8);
        \draw[-{latex}] (9) to (8);
        \draw[-{latex}, bend left = 30] (8) to (10);
        \draw[-{latex}] (10) to (9);

        \draw[-{latex}] (11) to (12);
        \draw[-{latex}] (13) to (12);
        \draw[-{latex}, bend left = 30] (12) to (14);
        \draw[-{latex}] (14) to (13);
        \draw[-{latex}] (14) to (15);

        \draw[-{latex}] (17) to (18);
        \draw[-{latex}] (19) to (18);
        \draw[-{latex}, bend left = 30] (18) to (20);
        \draw[-{latex}] (20) to (19);
        \draw[-{latex}] (20) to (21);
        \draw[-{latex}] (21) to (22);
        \draw[-{latex}] (22) to (23);
    \end{tikzpicture}
    \caption{The sequence $\{\mu_3(Q_\infty)\}$ from Figure \ref{fig:inf1}.}
    \label{fig:mu2}
\end{figure}

\section{Infinite Reddening Sequences}

In this section we start by defining the notion of \textbf{\textit{reddening sequences for an infinite quiver $Q_\infty$}} as an infinite or bi-infinite mutation sequence $\sigma$ such that the finite subsequence $\sigma|_{V(Q_i)}$ is a reddening sequence for $Q_i$ for all $i\geq 1$. 

The reddening sequence for $Q_\infty$ is a universal reddening sequence for the $Q_i$, in the sense that it encodes a family of reddening sequences for all of the finite quivers. Here are a few examples to help illustrate this concept.  

\begin{figure}
    \centering
    \scriptsize
        \begin{tikzpicture}[scale=0.3]
        \node (1) at (0,8) {$0$};
        \node (2) at (0,6) {$-1$};
        \node (3) at (3,6) {$0$};
        \node (4) at (6,6) {$1$};        
        \node (5) at (0,4) {$-2$};
        \node (6) at (3,4) {$-1$};
        \node (7) at (6,4) {$0$};
        \node (8) at (9,4) {$1$};
        \node (9) at (12,4) {$2$};
        \node (10) at (-4,8) {$Q_1$:};
        \node (11) at (-4,6) {$Q_2$:};
        \node (12) at (-4,4) {$Q_3$:};
        \node (13) at (0,0) {$\cdots$};
        \node (14) at (3,0) {$-2$};
        \node (15) at (6,0) {$-1$};
        \node (16) at (9,0) {$0$};
        \node (17) at (12,0) {$1$};
        \node (18) at (15,0) {$2$};
        \node (19) at (18,0) {$\cdots$};
        \node (20) at (9,2) {$\vdots$};
        \node (21) at (-4,0) {$Q_\infty$:};        
        \node (22) at (22,8) {$\sigma|_{V(Q_1)}=(0)$};
        \node (23) at (22,6) {$\sigma|_{V(Q_2)}=(-1,0,1)$};
        \node (24) at (22,4) {$\sigma|_{V(Q_3)}=(-2,-1,0,1,2)$};

        \draw[-{latex}] (2) to (3);
        \draw[-{latex}] (3) to (4);

        \draw[-{latex}] (5) to (6);
        \draw[-{latex}] (6) to (7);
        \draw[-{latex}] (7) to (8);
        \draw[-{latex}] (8) to (9);

        \draw[-{latex}] (13) to (14);
        \draw[-{latex}] (14) to (15);
        \draw[-{latex}] (15) to (16);
        \draw[-{latex}] (16) to (17);
        \draw[-{latex}] (17) to (18);
        \draw[-{latex}] (18) to (19);
        \end{tikzpicture}
    \caption{An infinite quiver with reddening sequence \\ $\sigma=(\dots,-2,-1,0,1,2,\dots)$.}
    \label{fig:red0}
\end{figure}

\begin{example} 
     Consider the infinite quiver in Figure \ref{fig:red0}. Here, we have the reddening sequence $\sigma=(\dots,-3,-2,-1,0,1,2,3,\dots)$, illustrating the need to allow a bi-infinite reddening sequence for infinite quivers. Notice that the infinite quiver itself is the same as in Figure \ref{fig:inf1}, but the sequence of infinite quivers are embedded differently. It allows the reddening sequence for each $Q_i$ to look like $\sigma|_{V(Q_i)}=(-i+1,\dots,i-1)$.  
\end{example}

\begin{example}
    Similarly, a reddening sequence for an infinite quiver is not required to be bi-infinite, as seen in Figure \ref{fig:red1}. Here, the sequence of finite quivers is embedded in such a way that allows for a singly infinite reddening sequence. The infinite quiver in Figure \ref{fig:red1} can be considered a star with two paths from the center. This form of infinite quiver can be generalized into stars with a finite amount of paths from the center. Consider Figure \ref{fig:red2}, a star with three paths from the center. The reddening sequence for this infinite quiver resembles that of the quiver from Figure \ref{fig:red1} but with an added element for each layer. This general form can be used for a star with any finite number of paths from the center. 
\end{example}

\begin{figure}
    \centering
    \scriptsize
        \begin{tikzpicture}[scale=0.3]
        \node (1) at (0,8) {$0$};
        \node (2) at (0,6) {$-1$};
        \node (3) at (3,6) {$0$};
        \node (4) at (6,6) {$1$};        
        \node (5) at (0,4) {$-2$};
        \node (6) at (3,4) {$-1$};
        \node (7) at (6,4) {$0$};
        \node (8) at (9,4) {$1$};
        \node (9) at (12,4) {$2$};
        \node (10) at (-4,8) {$Q_1$:};
        \node (11) at (-4,6) {$Q_2$:};
        \node (12) at (-4,4) {$Q_3$:};
        \node (13) at (0,0) {$\cdots$};
        \node (14) at (3,0) {$-2$};
        \node (15) at (6,0) {$-1$};
        \node (16) at (9,0) {$0$};
        \node (17) at (12,0) {$1$};
        \node (18) at (15,0) {$2$};
        \node (19) at (18,0) {$\cdots$};
        \node (20) at (9,2) {$\vdots$};
        \node (21) at (-4,0) {$Q_\infty$:};        
        \node (22) at (22,8) {$\sigma|_{V(Q_1)}=(0)$};
        \node (23) at (22,6) {$\sigma|_{V(Q_2)}=(0,-1,1)$};
        \node (24) at (22,4) {$\sigma|_{V(Q_3)}=(0,-1,1,-2,2)$};

        \draw[-{latex}] (3) to (2);
        \draw[-{latex}] (3) to (4);

        \draw[-{latex}] (6) to (5);
        \draw[-{latex}] (7) to (6);
        \draw[-{latex}] (7) to (8);
        \draw[-{latex}] (8) to (9);

        \draw[-{latex}] (14) to (13);
        \draw[-{latex}] (15) to (14);
        \draw[-{latex}] (16) to (15);
        \draw[-{latex}] (16) to (17);
        \draw[-{latex}] (17) to (18);
        \draw[-{latex}] (18) to (19);
        \end{tikzpicture}
    \caption{An infinite quiver with reddening sequence \\ $\sigma=(0,-1,1,-2,2,-3,3,\dots)$.}
    \label{fig:red1}
\end{figure}

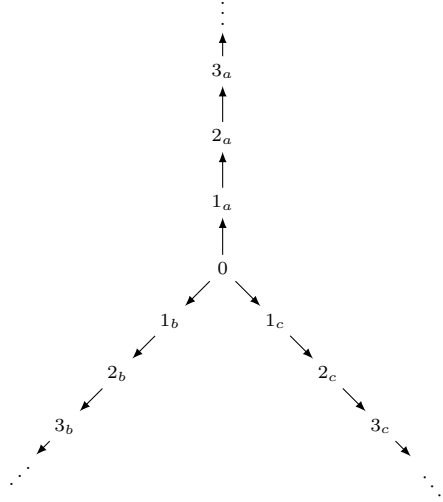
\begin{figure}
    \centering
    \scriptsize
    \begin{tikzpicture}[scale=0.35]
        \centering
        \node (0) at (0,0) {$0$};
        \node (1a) at (0,2.5) {$1_a$};
        \node (1b) at (-2,-2) {$1_b$};
        \node (1c) at (2,-2) {$1_c$};
        \node (2a) at (0,5) {$2_a$};
        \node (2b) at (-4,-4) {$2_b$};
        \node (2c) at (4,-4) {$2_c$};
        \node (3a) at (0,7.5) {$3_a$};
        \node (3b) at (-6,-6) {$3_b$};
        \node (3c) at (6,-6) {$3_c$};
        \node (4a) at (0,10) {$\vdots$};
        \node (4b) at (-8,-8) {\rotatebox[origin=c]{87}{$\ddots$}};
        \node (4c) at (8,-8) {\rotatebox[origin=c]{-7}{$\ddots$}};

        \draw[-{latex}] (0) to (1a);
        \draw[-{latex}] (1a) to (2a);
        \draw[-{latex}] (2a) to (3a);
        \draw[-{latex}] (3a) to (4a);

        \draw[-{latex}] (0) to (1b);
        \draw[-{latex}] (1b) to (2b);
        \draw[-{latex}] (2b) to (3b);
        \draw[-{latex}] (3b) to (4b);

        \draw[-{latex}] (0) to (1c);
        \draw[-{latex}] (1c) to (2c);
        \draw[-{latex}] (2c) to (3c);
        \draw[-{latex}] (3c) to (4c);
    \end{tikzpicture}
    \caption{An infinite quiver with infinite reddening sequence $\sigma=(0,1_a,1_b,1_c,2_a,2_b,2_c,3_a,3_b,3_c,\dots)$.}
    \label{fig:red2}
\end{figure}

We now consider a result which states the triangular extension construction preserves the existence of reddening sequences and maximal green sequences.
Let $Q_1$ and $Q_2$ be quivers. A \textbf{\emph{triangular extension}} of $Q_1$ and $Q_2$ is any quiver $Q$ with
\[V(Q) = V(Q_1) \sqcup V(Q_2)\]
\[E(Q) = E(Q_1) \sqcup E(Q_2) \sqcup E\]
where $E$ is any set of arrows such that either we have either
\begin{itemize}
\item[] for any $i \to j \in E$ implies $i \in V(Q_1)$ and $j \in V(Q_2)$
\end{itemize}
or else
\begin{itemize}
\item[] for any $i \to j \in E$ implies $i \in V(Q_2)$ and $j \in V(Q_1)$.
\end{itemize}
That is, a triangular extension of quivers simply takes the disjoint union of the two quivers then adds additional arrows between the quivers with the condition that all arrows are directed from one quiver to the other.
An example of a triangular extension of quivers $Q_1$ and $Q_2$ where $V(Q_1) = \{1,2,3\}$ and $V(Q_2) = \{4,5,6\}$ is given in Figure~\ref{fig:directsum}.

The following is a result by Cao and Li that will be utilized to prove our main result on infinite quivers: 

\begin{lemma}[{\cite[Theorem 4.5, Remark 4.6]{cao-li}}]
If $Q_1$ and $Q_2$ are any two quivers which both admit reddening (maximal green) sequences, then any triangular extension of $Q_1$ and $Q_2$ admits a reddening (maximal green) sequence.
\label{lem:directsum}
\end{lemma}

\begin{figure}
\begin{tikzpicture}
\node (1) at (0,0.5) {$1$};
\node (2) at (1,-0.5) {$2$};
\node (3) at (-1,-0.5) {$3$};
\draw[-{latex}] (1) to (2);
\draw[-{latex}] (2) to (3);
\draw[-{latex}] (3) to (1);

\node (4) at (3,1.5) {$4$};
\node (5) at (3,0) {$5$};
\node (6) at (3,-1.5) {$6$};
\draw[-{latex}] (4) to (5);
\draw[-{latex}] (6) to (5);

\draw[-{latex}] (1) to (4);
\draw[-{latex}] (1) to (5);
\draw[-{latex}] (2) to (4);
\draw[-{latex}] (2) to[bend right=30] (5);
\draw[-{latex}] (2) to (5);
\draw[-{latex}] (2) to (6);
\end{tikzpicture}
\caption{A triangular extension of quivers.}
\label{fig:directsum}
\end{figure}
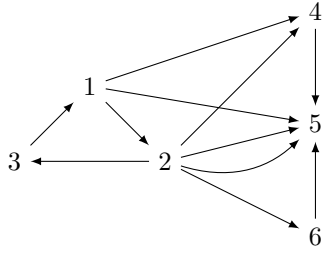

\begin{theorem}
Let $Q_\infty = \{Q_i\}$ be an infinite quiver. If the following criteria are met:

\begin{enumerate}
    \item $Q_i$ is a triangular extension of $Q_{i-1}$ and $Q_i \setminus Q_{i-1}$ for each $i\in \mathbb{N}$,
    \item $Q_i \setminus Q_{i-1}$ has a reddening sequence for each $i \in \mathbb{N}$, and
    \item $Q_1$ has a reddening sequence,
\end{enumerate}
 then $Q_\infty$ has a reddening sequence.
\end{theorem}

\begin{proof}
    Let $Q_\infty = \{ Q_i \}$ be an infinite quiver as in the theorem. Define $W_i := Q_i \setminus Q_{i-1}$. Also define $\tau_i$ as the reddening sequence for $W_i$ whose existence is criteria $(2)$. We will now construct the reddening sequence, $\sigma_\infty$, for $Q_\infty$ iteratively.\\

    Let $\sigma_1$ be the reddening sequence for $Q_1$ given by criteria $(1)$. Now we must define $\sigma_i$ in terms $\sigma_{i-1}$. For this there are two cases:\\
    
    \textbf{Case One:} $Q_i = Q_i \ \stackanchor{$\oplus$}{$\rightarrow$} \ W_i$.

    This is the case where all of the arrows connecting $Q_i$ to $W_i$ are outgoing. Then by Lemma \ref{lem:directsum} we have that $(\sigma_{i-1}, \tau_i)$ is a reddening sequence for $Q_i$.\\
   
   \textbf{Case One:} $Q_i = Q_i \ \stackanchor{$\oplus$}{$\leftarrow$} \ W_i$.

    This is the case where all of the arrows connecting $Q_i$ to $W_i$ are incoming. Then by Lemma \ref{lem:directsum} we have that $( \tau_i, \sigma_{i-1})$ is a reddening sequence for $Q_i$.\\
   
   The resulting bi-infinite sequence obtained by adding the $\tau_i$ to either the left or the right of the previous reddening sequence is a reddening sequence for $Q_\infty$. 
\end{proof}

Though this result seems restrictive, one can build $Q_\infty$ using many different sequences. In essence, what the result is saying is that if you can build the infinite quiver, $Q_\infty$, in such a way that all of the newly added parts of the quiver have reddening sequences of their own, and are triangular extensions of the previous quiver in the sequence then in fact you have a reddening sequence for the infinite quiver. For example this guarantees the infinite quiver in Figure \ref{triangles} a reddening sequence. 

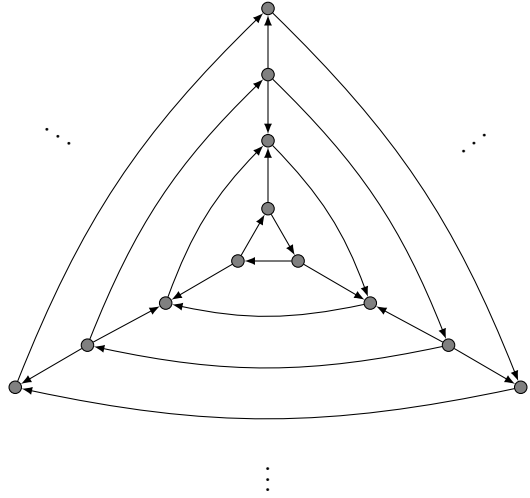
\begin{figure}
    \centering
    \begin{tikzpicture}[scale=0.8]
        \centering
        \node[draw,circle,fill=gray,scale=0.5] (1a) at (0,0.87) {};
        \node[draw,circle,fill=gray,scale=0.5] (1b) at (0.5,0) {};
        \node[draw,circle,fill=gray,scale=0.5] (1c) at (-0.5,0) {};
        
        \draw[-{latex}] (1a) to (1b);
        \draw[-{latex}] (1b) to (1c);
        \draw[-{latex}] (1c) to (1a);        
        
        \node[draw,circle,fill=gray,scale=0.5] (2a) at (0,2) {};
        \node[draw,circle,fill=gray,scale=0.5] (2b) at (1.7,-0.7) {};
        \node[draw,circle,fill=gray,scale=0.5] (2c) at (-1.7,-0.7) {};

        \draw[-{latex}, bend left = 12] (2a) to (2b);
        \draw[-{latex}, bend left = 12] (2b) to (2c);
        \draw[-{latex}, bend left = 12] (2c) to (2a);

        \node[draw,circle,fill=gray,scale=0.5] (3a) at (0,3.1) {};
        \node[draw,circle,fill=gray,scale=0.5] (3b) at (3,-1.4) {};
        \node[draw,circle,fill=gray,scale=0.5] (3c) at (-3,-1.4) {};

        \draw[-{latex}, bend left = 12] (3a) to (3b);
        \draw[-{latex}, bend left = 12] (3b) to (3c);
        \draw[-{latex}, bend left = 12] (3c) to (3a);

        \node[draw,circle,fill=gray,scale=0.5] (4a) at (0,4.2) {};
        \node[draw,circle,fill=gray,scale=0.5] (4b) at (4.2,-2.1) {};
        \node[draw,circle,fill=gray,scale=0.5] (4c) at (-4.2,-2.1) {};

        \draw[-{latex}, bend left = 12] (4a) to (4b);
        \draw[-{latex}, bend left = 12] (4b) to (4c);
        \draw[-{latex}, bend left = 12] (4c) to (4a);

        \draw[-{latex}] (1a) to (2a);
        \draw[-{latex}] (3a) to (2a);
        \draw[-{latex}] (3a) to (4a);
        \draw[-{latex}] (1b) to (2b);
        \draw[-{latex}] (3b) to (2b);
        \draw[-{latex}] (3b) to (4b);
        \draw[-{latex}] (1c) to (2c);
        \draw[-{latex}] (3c) to (2c);
        \draw[-{latex}] (3c) to (4c);
        
        \node (dot1) at (-3.5,2.2) {\rotatebox[origin=c]{5}{$\ddots$}};
        \node (dot2) at (3.3,2) {\rotatebox[origin=c]{75}{$\ddots$}};
        \node (dot3) at (0,-3.5) {$\vdots$};
    \end{tikzpicture}
    \caption{An infinite quiver $Q^\infty$.}
    \label{triangles}
\end{figure}

\section{Future Work}\label{sec:futurework}

There are many natural questions that arise from the definitions we have provided for infinite quivers and mutation on these mathematical objects. We present here a few of these problems: 

Consider the surface on the left of Figure \ref{fig:future}, a triangulated surface with one puncture and one limit arc. Under the topological framework provided by Çanaçki and Felikson \cite{C-F}, this type of surface is allowed. The quiver that we would naturally think arises from this surface is seen on the right of the same figure. However, under our definition of infinite quivers, there seems to be no way to build the sequence of finite quivers that results in the infinite quiver in Figure \ref{fig:future}. We tried to tweak our definition of an infinite quiver to be the sequence of finite quivers $\{Q_i\}$ such that $Q_i$ is a graph minor of $Q_{i+1}$. This would allow the infinite quiver to be the limit of a sequence of directed cycles of increasing order, however this presented some problems when we tried to find the correct sequence of embeddings $\{\tau_i\}$, so we did not change our definition.  

\begin{question}
    Is there a construction analogous to the sequence of induced subgraph definition of infinite quiver, that instead uses an embedded sequence of graph minors which gives the above example as a limit of an infinite sequence of finite quivers? 
\end{question}

\begin{figure}[h]
\centering
\begin{minipage}{0.45\textwidth}
\centering
\begin{tikzpicture}[scale=2.75]
  \draw[thick,black] (0,0) circle (1);
  \fill[cyan!20] (0,0) circle (1);

  \def\n{20}
  \def\shorta{17}
  \def\shortb{18}
  \def\dotted{16}

  \foreach \i in {0,...,19}{
    \pgfmathsetmacro{\a}{360/\n*\i}
    \ifnum\i=\shorta\else\ifnum\i=\shortb\else\ifnum\i=\dotted\else
      \draw[thick,blue](0,0)--({cos(\a)},{sin(\a)});
      \fill[black]({cos(\a)},{sin(\a)})circle(0.02);
    \fi\fi\fi
  }

  \pgfmathsetmacro{\aA}{360/\n*\shorta}
  \pgfmathsetmacro{\aB}{360/\n*\shortb}
  \pgfmathsetmacro{\aC}{360/\n*\dotted}

  \draw[thick,blue](0,0)--({0.8*cos(\aA)},{0.8*sin(\aA)});
  \draw[thick,blue](0,0)--({0.8*cos(\aB)},{0.8*sin(\aB)});

  \pgfmathsetmacro{\r}{0.82}
  \foreach \f in {0.25,0.5,0.75}{
    \pgfmathsetmacro{\ang}{\aA+\f*(\aB-\aA)}
    \fill[black]({\r*cos(\ang)},{\r*sin(\ang)})circle(0.01);
  }

  \draw[thick,blue,dashed](0,0)--({cos(\aC)},{sin(\aC)});
  \draw[thin,black]({cos(\aC)},{sin(\aC)})circle(0.02);

  \fill[black](0,0)circle(0.03);
\end{tikzpicture}
\end{minipage}
\hfill
\begin{minipage}{0.45\textwidth}
\centering
\begin{tikzpicture}[scale=2.75]
  \node[draw,circle,fill=gray,scale=0.5] (v0)  at (1,0) {};
  \node[draw,circle,fill=gray,scale=0.5] (v1)  at (0.866,0.5) {};
  \node[draw,circle,fill=gray,scale=0.5] (v2)  at (0.5,0.866) {};
  \node[draw,circle,fill=gray,scale=0.5] (v3)  at (0,1) {};
  \node[draw,circle,fill=gray,scale=0.5] (v4)  at (-0.5,0.866) {};
  \node[draw,circle,fill=gray,scale=0.5] (v5)  at (-0.866,0.5) {};
  \node[draw,circle,fill=gray,scale=0.5] (v6)  at (-1,0) {};
  \node[draw,circle,fill=gray,scale=0.5] (v7)  at (-0.866,-0.5) {};
  \node[draw,circle,fill=gray,scale=0.5] (v8)  at (-0.5,-0.866) {};
  \node[draw,circle,fill=gray,scale=0.5] (v9)  at (0,-1) {};
  \node[draw,circle,fill=gray,scale=0.5] (v10) at (0.5,-0.866) {};
  \node[draw=none,circle,fill=white,scale=2.25] (v11) at (0.866,-0.5) {};
  \node (dots) at (0.79,-0.485) {\rotatebox[origin=c]{94}{\textbf{$\ddots$}}};

  \draw[-{latex}] (v0) to (v1);
  \draw[-{latex}] (v1) to (v2);
  \draw[-{latex}] (v2) to (v3);
  \draw[-{latex}] (v3) to (v4);
  \draw[-{latex}] (v4) to (v5);
  \draw[-{latex}] (v5) to (v6);
  \draw[-{latex}] (v6) to (v7);
  \draw[-{latex}] (v7) to (v8);
  \draw[-{latex}] (v8) to (v9);
  \draw[-{latex}] (v9) to (v10);
  \draw[-{latex}] (v10) to (v11);
  \draw[-{latex}] (v11) to (v0);
\end{tikzpicture}
\end{minipage}

\caption{The left shows a surface with one puncture and one limit arc, and the right shows the infinite quiver that would naturally arise from such a surface.}
\label{fig:future}
\end{figure}
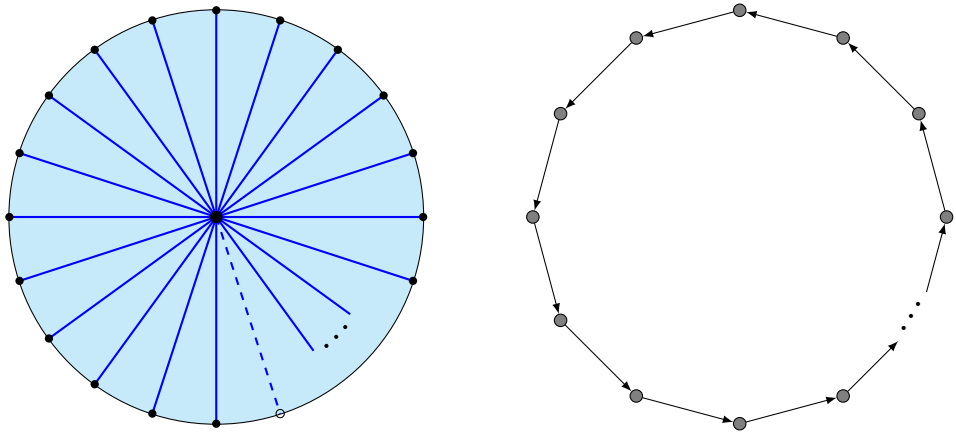

Another area of research that is left open is whether the reddening sequences explored in the previous section are invariant under mutation. In the finite case Muller \cite{Muller} utilizes scattering diagrams to show that if $Q$ is a finite quiver that admits a reddening sequence, so does any mutation equivalent quiver $Q'$. The hope would be that a similar result is true in the case of infinite quivers, but we cannot utilize the notion of scattering diagrams in this case. What one can show is that at every finite subquiver in the infinite quiver we have a reddening sequence, which will guarantee the existence of a reddening sequence for its mutation equivalent counterpart, but what we cannot show is whether that sequence of finite reddening sequences is in fact compatible as an infinite reddening sequence.

\begin{question}
    Can the existence of an infinite reddening sequence for $Q$ be used to build an infinite reddening sequence for a mutation equivalent $Q'$? If not can a counter example be found where two infinite quivers $Q$ and $Q'$ are mutation equivalent yet only one of them admits a reddening sequence. Or in other words, is the existence of an infinite reddening sequence for an infinite quiver invariant under mutation? 
\end{question}

The authors would like to acknowledge the thoughtful conversation and help from John Machacek on this paper, as it has greatly helped us develop these ideas.


\bibliographystyle{plain}
\bibliography{paper}

\end{document}